\documentclass[a4paper,11pt]{amsart}

\usepackage{amsfonts,amsthm,amssymb,amsmath,amscd,ascmac}
\usepackage[all]{xy}

\newtheorem{lem}{Lemma}[section]
\newtheorem{prop}[lem]{Proposition}
\newtheorem{cor}[lem]{Corollary} 
\newtheorem{rem}[lem]{Remark} 
\newtheorem{tm}[lem]{Theorem} 
 
\newcommand{\ZZ}{\mathbb{Z}}
\newcommand{\QQ}{\mathbb{Q}}

\newcommand{\CC}{\mathbb{C}}
\newcommand{\PP}{\mathbb{P}}

\newcommand{\GL}{\mathrm{GL}}
\newcommand{\Aut}{\mathrm{Aut}}

\newcommand{\w}{\omega}

\numberwithin{equation}{section}
\begin{document}
\title{On cubic fourfolds with an inductive structure}
\date{}
\author{Kenji Koike, Yamanashi University}  
%%%%%%%%%%%%%%%%%%%%%%%%%%%%%%%%%%%%%%%%%%%%%%%%% ABSTRACT
\begin{abstract}
 We study the number of planes for four dimensional projective hypersurfaces 
which has so-called inductive structure. We also determine transcendental lattices 
for cubic fourfolds of this type. 
\end{abstract}
%%%%%%%%%%%%%%%%%%%%%%%%%%%%%%%%%%%%%%%%%%%%%%%%%%%%%%%%%%
\maketitle
%%%%%%%%%%%%%%%%%%%%%%%%%%%%%%%%%%%%%%%%%%%%%%%%%%%%%%%%%%
\section{Introduction} 
Algebraic surfaces in $\PP^3(\CC)$ defined by a homogeneous equation
\begin{align} \label{surface}
 f_1(x, y) = f_2(z, w)
\end{align}
have been well studied as examples of surfaces with many lines. For example, 
a smooth quartic surface over $\CC$ contains at most 64 lines (Segre, Ram and 
Sch\"{u}tt, \cite{RS15}) and this bound is realized by Schur's quartic 
\[
 x(x^3 - y^3) = z(z^3 - w^3).
\]  
Let $f(u,v)$ be a Klein's invariant polynomial   
\[
 f(u, v) = (u^{30} + v^{30}) + 522(u^{25} v^5 - u^5 v^{25}) 
- 10005(u^{20} v^{10} + u^{10} v^{20})
\]
for the icosahedral group. Then the surface $f(x, y) = f(z, w)$ of degree 
$30$ has $2700$ lines (Boissi\'ere and Sarti, \cite{BS07}). In general, we have
%%%%%%%%%%%%%%%%%%%%%%%%%%%%%%%%%%%%%%%%%%%%%%%%%%%%%%%%%%%%%%%%%
%%%%%%%%%%%%%%%%%%%%%%%%%%%%%%%%%%%%%%%%%%%%%%%%%%%%%%%%% theorem
\begin{prop}[\cite{BS07, CHM95}]
Let $S$ be a smooth surface of degree $d$ defined by (\ref{surface}). 
 The number of lines on $S$ is exactly $d(d + \alpha_d)$, where $\alpha_d$ is 
the order of the group of isomorphisms of $\PP^1(\CC)$ mapping zeroes of $f_1$ to 
zeroes of $f_2$. 
\end{prop}
%%%%%%%%%%%%%%%%%%%%%%%%%%%%%%%%%%%%%%%%%%%%%%%%%%%%%%%%
In this paper, we study four dimensional hypersurfaces 
$X$ defined by  a homogeneous equation
\begin{align} \label{eqX}
 F_1(x_0, x_1, x_2) = F_2(y_0, y_1, y_2)
\end{align}
with projective coordinates $[x_0: x_1: x_2: y_0: y_1: y_2] \in \PP^5(\CC)$. 
Note that $X$ is smooth if and only if plane curves
\begin{align} \label{eqC}
 C_1 : F_1(x_0, x_1, x_2) = y_0 = y_1 = y_2 = 0, \qquad 
C_2 : x_0 = x_1 = x_2 = F_2(y_0, y_1, y_2) = 0
\end{align}
are smooth (For simplicity, we will use coordinates $[x_0:x_1:x_2]$ to represent 
points on $C_1$). 
To state our results, we introduce a few notations and terminologies. 
We say that a homogeneous polynomial $F(x_0,x_1,x_2)$ is equivalent to $G(x_0,x_1,x_2)$ 
and write $F \sim G$, if there is $g \in \GL_3(\CC)$ such that $F(g \cdot x) = G(x)$. 
For a homogeneous polynomial $F(x_0,x_1,x_2)$, we define the group of automorphisms 
\[
 \Aut(F) = \{ g \in \GL_3(\CC) \ | \ F(gx) = F(x)\}.
\]
The group $\Aut(F)$ contains $\mu_d = \{ \zeta \in \CC \ | \ \zeta^d = 1\}$ 
as scalar matrices, where $d = \deg F$.  
Next, let $C$ be a plane curve. We call a smooth point $P \in C$ a $k$-{\it flex point} 
if the tangent line at $P$ and $C$ intersect with intersection multiplicity $k \geq 3$. 
We will show  the following Proposition in Section 2.
%%%%%%%%%%%%%%%%%%%%%%%%%%%%%%%%%%%%%%%%%%%%%%%%%%%%%%%%%%%%%%%%%%%%%%%%%%%%%%
%%%%%%%%%%%%%%%%%%%%%%%%%%%%%%%%%%%%%%%%%%%%%%%%%%%%%%%%%%%%%%%%%%%%%% theorem
\begin{prop} \label{main-th1}
Let $X \subset \PP^5(\CC)$ be a smooth hypersurface of degree $d \geq 3$ 
defined by {\rm (\ref{eqX})}, and $\nu_i$ be the number of $d$-flex points. 
\\
{\rm (1)} \ If $F_1 \sim F_2$, the number of planes in $X$ is $d \nu_1^2 + |\Aut(F_1)|$.
\\
{\rm (2)} \ If $F_1 \not\sim F_2$, the number of planes in $X$ is $d \nu_1 \nu_2$. 
\end{prop} 
%%%%%%%%%%%%%%%%%%%%%%%%%%%%%%%%%%%%%%%%%%%%%%%%%%%%%%%%%%%%%%%%%%%%%%
In the cases of $d = 3$, we have
%%%%%%%%%%%%%%%%%%%%%%%%%%%%%%%%%%%%%%%%%%%%%%%%%%%%%%%%%%%%%%%%%%%%%
%%%%%%%%%%%%%%%%%%%%%%%%%%%%%%%%%%%%%%%%%%%%%%%%%%%%%%%%% corollary
\begin{cor} \label{cor1} 
Let $X$ be a smooth cubic fourfold defined by the equation {\rm (\ref{eqX})}.
\\
{\rm (1)} \ If $F_1 \sim F_2$, the number of planes in $X$ is 
$\begin{cases}
  243 + 162 \quad (j = 0) \\ 
  243 + 108 \quad (j = 1728) \\
  243 + 54 \quad (j \ne 0, 1728)
\end{cases},$ 
where $j$ is the $j$-invariant of $C_i$ as an elliptic curve.  
\\
{\rm (2)} \ If $F_1 \not\sim F_2$, then $X$ contains $243$ planes.
\end{cor}
%%%%%%%%%%%%%%%%%%%%%%%%%%%%%%%%%%%%%%%%%%%%%%%%%%%%%%%%%%%%%%%%%%%%%%% 
Note that $X$ is isomorphic to  the Fermat cubic fourfold if $j = 0$. 
Recently, there is a progress in the number of planes in cubic fourfolds 
by Degtyarev, Itenberg and J. C. Ottem.
\begin{tm}[\cite{DIO21}]
Let $X \subset \PP^5(\CC)$ be a smooth cubic fourfold. Then, either $X$ has 
at most 350 planes, or, up to projective equivalence, $X$ is
\\
{\rm (1)} \ the Fermat cubic (with 405 planes), or
\\
{\rm (2)} \ the Clebsch Segre cubic (with 357 planes), or
\\
{\rm (3)} \ the cubic fourfold with 351 planes.
\end{tm}
%%%%%%%%%%%%%%%%%%%%%%%%%%%%%%%%%%%%%%%%%%%%%%%%%%%%%%%%%%%%%%%%%%%%%%%%%
The Approach in \cite{DIO21} is based on the lattice theorey and the Torelli 
theorem, and it seems that the cubic fourfold with 351 planes is not given  
explicitly. It is just the second case of (1) in Corollary \ref{cor1}.
\\ \indent 
%%%%%%%%%%%%%%%%%%%%%%%%%%%%%%%%%%%%%%%%%%%%%%%%%%%%%%%%%%%%%%%%%%%%%%%%%
The second purpose of this paper is to determine transcendental lattices for 
cubic fourfolds of this type. The middle cohomology for hypersurfaces of this type 
was studied by Shioda and Katsura using ``inductive structure''
(\cite{SK79}, \cite{Sh79}). 
Let us consider smooth algebraic suraces of degree $d$
%%%%%%%%%%%%%%%%%%%%%%%%%%%%%%%%%%%%%%%%%%%%
\begin{align}
  S_1 : F_1(x_0, x_1, x_2) = x_3^d, \qquad 
  S_1 : F_2(y_0, y_1, y_2) = y_3^d
\end{align}
%%%%%%%%%%%%%%%%%%%%%%%%%%%%%%%%%%%%%%%%%%%%
that are cyclic coverings of $\PP^2$ branched along $C_1$ and $C_2$ respectively. 
We have a dominant rational map $S_1 \times S_2 \dashrightarrow X$ defined by
%%%%%%%%%%%%%%%%%%%%%%%%%%%%%%%%%%%%%%%%%%%%
\begin{align*} 
[x_0: \cdots : x_3] \times [y_0: \cdots : y_3] \mapsto 
[y_3x_0 : y_3x_1 : y_3 x_2 : x_3 y_0 \ x_3 y_1 : x_3 y_2].  
\end{align*}
%%%%%%%%%%%%%%%%%%%%%%%%%%%%%%%%%%%%%%%%%%%%
\begin{tm}[\cite{SK79}, \cite{Sh79}]
 This rational map induces an isomorphism 
%%%%%%%%%%%%%%%%%%%%%%%%
  \begin{align*}
   [H_{prim}^2(S_1, \CC) \otimes H_{prim}^2(S_2, \CC)]^{\mu_d} \oplus
   [H^1(C_1, \CC) \otimes H^1(C_2, \CC)](1) \xrightarrow{\sim} H_{prim}^4(X, \CC)
  \end{align*}
%%%%%%%%%%%%%%%%%%%%%%%
where $\mu_d$ acts as 
%%%%%%%%%%%%%%%%%%%%%%%
\begin{align*} 
[x_0: y_1 : x_2 : x_3] \times [y_0: y_1 : y_2 : y_3] \mapsto 
[x_0: y_1 : x_2 : \zeta x_3] \times [y_0: y_1 : y_2 : \zeta y_3].  
\end{align*}
\end{tm}
%%%%%%%%%%%%%%%%%%%%%%%
Namely, the Hodge structure of $X$ is determined by that of $C_i$ and $S_i$.   
In a different context, Beauville noted
%%%%%%%%%%%%%%%%%%%%%%%%%%%%%%%%%%%%%%%%%%%%%%%%%%%%%%%%%%%%%%%%%%%%%%%%%%%%%%%%%%%%
\begin{prop}[\cite{Be14}]
Let $F$ be a cubic form in 3 variables, such that the curve $F(x,y,z) = 0$ in $\PP^2$ 
is an elliptic curve with complex multiplication. Let $X$ be the cubic fourfold defined 
by $F(x_0, x_1, x_2) = F(y_0, y_1, y_2)$ in $\PP^5$. Then $H^4(X, \ZZ)_{alg}$ has 
maximal rank $h^{2,2}(X)$. 
\end{prop}
%%%%%%%%%%%%%%%%%%%%%%%%%%%%%%%%%%%%%%%%%%%%%%%%%%%%%%%%%%%%%%%%%%%%%%%%%%%%%%%%%
Computing intersection numbers of transcendental cycles coming from $C_1 \times C_2$, 
we determine the transcwndemtal lattice for cubic fourfolds of this type. 
Note that the integral Hodge conjecture is known for cubic fourfolds 
(\cite{V07}). In Section 3 and 4, we show
%%%%%%%%%%%%%%%%%%%%%%%%%%%%%%%%%%%%%%%%%%%%%%%%%%%%%%%%% 
\begin{tm} \label{main-th2}
Let $X$ be a smooth cubic fourfold defined by the equation {\rm (\ref{eqX})}.
We have a morphism of integral Hodge structures
%%%%%%%%%%%%%%
\begin{align*}
  \phi : [H^1(C_1, \ZZ) \otimes H^1(C_2, \ZZ)](1) 
\longrightarrow H^4(X,\ZZ)
\end{align*}
%%%%%%%%%%%%
such that $\phi(\alpha) \cdot \phi(\beta) = -3(\alpha \cdot \beta)$.  
For a general $X$, the transcendental lattice $T_X$ of $X$, that is, 
the orthogonal complement of $H^{2,2}(X) \cap H^4(X, \ZZ)$ is given by 
\[
T_X = \mathrm{Im} \, \phi = \mathrm{U}(3) \oplus \mathrm{U}(3).
\]  
\end{tm}
%%%%%%%%%%%%%%%%%%%%%%%%%%%%%%%%%%%%%%%%%%%%%%%%%%%%%%%%%%%%%%%%%%%%%%%%%%%%%
Let us recall the following formula by Shioda and Mitani (\cite{SM74}). 
For a  positive definite even latiice 
$T = \begin{bmatrix} 2a & b \\ b & 2c\end{bmatrix} (a,b,c \in \ZZ)$ of rank $2$, 
we have an Abelian surface $C_1 \times C_2$ whose transcendental lattice is $T$, 
where $C_i = \CC/(\ZZ + \ZZ \tau_i)$ and
\begin{align*}
  \tau_1 = \frac{-b + \sqrt{\Delta}}{2a}, \qquad
  \tau_2 = \frac{b + \sqrt{\Delta}}{2} \qquad (\Delta = b^2 -4ac).
\end{align*}
%%%%%%%%%%%%%%%%%%%%%%%%%%%%%%%%%%%%%%%%%%%%%%%%%%%%%%%%%%%%%%%%%%%%%%%%%%%%%
\begin{cor}
Let $F_i(x_0, x_1, x_2) = 0$ be the cubic equation of the above elliptic 
curve $C_i$.  Then the corresponding cubic fourfold $X$ defined by 
$F_1(x) = F_2(y)$ has the transcendental lattice 
$T(-3) =\begin{bmatrix} -6a & -3b \\ -3b & -6c\end{bmatrix}$.
\end{cor}
%%%%%%%%%%%%%%%%%%%%%%%%%%%%%%%%%%%%%%%%%%%%%%%%%%%%%%%%%%%%%%%%%%%%%%%%%%%%%
For example, we have $\tau_1 = \tau_2 = \sqrt{-1}$ for $a = c = 1, b=0$.  
The corresponding cubic fourfold $X$ (the case of $j=1728$ in Corollary \ref{cor1}) 
has the transcendental lattice $\begin{bmatrix} -6 & 0 \\ 0 & -6\end{bmatrix}$, 
which was computed as  the transcendental lattice for the cubic fourfold 
with 351 planes in \cite{DIO21}. 
%%%%%%%%%%%%%%%%%%%%%%%%%%%%%%%%%%%%%%%%%%%%%%%%%%%%%%%%%%%%%%%%%%%%%%%%%%%%%
\section{number of planes} 
%%%%%%%%%%%%%%%%%%%%%%%%%%%%%%%%%%%%%%%%%%%%%%%%%%%%%%%%%%%%%%%%%%%%%%%%%%%%
Let $S$ be a plane in $\PP^5(\CC)$ defined by a system of linear equations 
\begin{align} \label{plane-eq}
\begin{cases}
a_{00} x_0 + a_{01} x_1 + a_{02} x_2 =  b_{00} y_0 + b_{01} y_1 + b_{02} y_2 \\
a_{10} x_0 + a_{11} x_1 + a_{12} x_2 =  b_{10} y_0 + b_{11} y_1 + b_{12} y_2 \\
a_{20} x_0 + a_{21} x_1 + a_{22} x_2 =  b_{20} y_0 + b_{21} y_1 + b_{22} y_2
\end{cases}
\end{align}
of rank 3, which is denoted by $Ax = By$ where $A = [a_{ij}]$ and $B = [b_{ij}]$. 
Let $X$ be a smooth hypersurface in $\PP^5(\CC)$ defined by (\ref{eqX}).
%%%%%%%%%%%%%%%%%%%%%%%%%%%%%%%%%%%%%%%%%%%%%%%%%%%%%%%%%%%%%%%%%%%%%%%%%%%%%%%%%
%%%%%%%%%%%%%%%%%%%%%%%%%%%%%%%%%%%%%%%%%%%%%%%%%%%%%%%%%%% Lemma 
\begin{lem}
We assume that $\deg X \geq 2$ and $S \subset X$. Then, we have \\
{\rm (1)} \ $\mathrm{rank} \, A, \ \mathrm{rank} \, B \geq 2$,
\\
{\rm (2)} \ $\mathrm{rank} \, A = 3 
\quad \Leftrightarrow \quad \mathrm{rank} \, B = 3 
\quad \Rightarrow \quad F_1 \sim F_2$, 
\\
{\rm (3)} \ $\mathrm{rank} \, A = 2 
\quad \Leftrightarrow \quad \mathrm{rank} \, B = 2$. 
\end{lem}
%%%%%%%%%%%%%%%%%
\begin{proof} 
(1) \ If $\mathrm{rank} \, B = 0$, then we have $b_{i,j} = 0$ for all $i, j$
and the system (\ref{plane-eq}) is equivalent to $x_0 = x_1 = x_2 = 0$. 
By the assumption $S \subset X$, we have
\[
 F_1(0,0,0) = F_2(y_0, y_2, y_3)
\]
and this is not the case. 
Let us assume that $\mathrm{rank} \, B = 1$. Then, the system (\ref{plane-eq}) is 
reduced to
\begin{align*}
 \ell_0(x_0, x_1, x_2) &= m_0(y_0, y_1, y_2) \\
 \ell_1(x_0, x_1, x_2) &= 0 \\
 \ell_2(x_0, x_1, x_2) &= 0
\end{align*}  
where $\ell_i$ and $m_0$ are linear forms. By the assumption 
$S \subset X$, there exist homogeneous polynomials $q_i$ of degree $d-1$ such that 
\begin{align*}
  F_1(x) - F_2(y) = q_0(x,y) (\ell_0(x) - m_0(y)) 
+ q_1(x,y)\ell_1(x) + q_2(x,y) \ell_2(x)
\end{align*}
Putting $x_0 = x_1 = x_2 = 0$, we have $F_2(y) = q_0(0,y) m_0(y)$ and 
this contradicts irreducibility of $F_2$. As above, we have $\mathrm{rank} \, B \geq 2$ 
and the same is true for $A$. 
\\
(2) \ Let us assume that $\mathrm{rank} \, B = 3$. 
The plane $S$ is given by $B^{-1}Ax = y$ and we have 
an identity $F_1(x) = F_2(B^{-1}Ax)$ as a polynomial of $x_0, x_1, x_2$.
 Since $X$ is smooth, $F_1(x) = 0$ is a smooth plane curve   
and so is $F_2(B^{-1}Ax) = 0$. We see that $B^{-1}A \in \GL_3(\CC)$ 
(note that $\deg X \ne 1$) and $F_1 \sim F_2$. 
The same is true for the case $\mathrm{rank} \, A = 3$.
\\
(3) \ By (1) and (2),  we see that $\mathrm{rank} \, B = 2$ if and only if 
$\mathrm{rank} \, A = 2$. 
\end{proof}
%%%%%%%%%%%%%%%%%%%%%%%%%%%%%%%%%%%%%%%%%%%%%%%%%%%%%%%%%%%%%%%%%%%
In the following, we say that {\it a plane $S$ is of rank $k$} if
 $\mathrm{rank} \, A = k$.  
%%%%%%%%%%%%%%%%%%%%%%%%%%%%%%%%%%%%%%%%%%%%%%%%%%%%%%%%%%%%% Lemma
\begin{lem}
If $F_1 \sim F_2$, there exist exactly $|\Aut(F_1)|$ planes of rank $3$ in $X$. 
\end{lem}
%%%%%%%%%%%%%%%%%%%%%%%%%%%%%%%%%%%%%%%%%%%%%%%%%%%%%%%%%%%% proof
\begin{proof}
We may assume that $F_1 = F_2$. If $Ax = By$ is a plane of rank 3 in $X$, 
we have $F_1(x) = F_1(B^{-1}Ax)$ and $B^{-1}A \in \Aut(F_1)$. Conversely, $y = gx$ 
is a plane of rank 3 in $X$ for any $g \in \Aut(F_1)$.   
\end{proof}
%%%%%%%%%%%%%%%%%%%%%%%%%%%%%%%%%%%%%%%%%%%%%%%%%%%%%%%%%%%%%%%%%%%%%%%%%%%%%%%%% 
Before we consider the number of planes of rank 2, we note that
%%%%%%%%%%%%%%%%%%%%%%%%%%%%%%%%%%%%%%%%%%%%%%%%%%%%%%%%%%% Lemma
\begin{lem} \label{Lemma-div}
 Let $\phi_1(x_0, x_1)$ and $\phi_2(y_0, y_1)$ be homogeneous polynomials of 
degree $d$. If $\phi_1 - \phi_2$ is devided by $x_1 - y_1$, 
then $\phi_1 - \phi_2$ is a scalar multiple of $x_1^d - y_1^d$. 
\end{lem}
%%%%%%%%%%%%%%%%%%%%%%%%%%%%%%%%%%%%%%%%%%%%%%%%%%%%%%%%%%%% proof
\begin{proof}
We have
\begin{align*}
 \phi_1 - \phi_2 \equiv 0 \mod x_1 - y_1 
\qquad &\Leftrightarrow \qquad
 \phi_1(x_0, x_1) - \phi_2(y_0, x_1) = 0 \quad \text{in} \ \CC[x_0,x_1,y_0], 
\end{align*}
and a binary forms $\phi_1(x_0, x_1)$ coincide with $\phi_2(y_0, x_1)$ if and only if 
\[
 \phi_1(x_0, x_1) = \phi_2(y_0, x_1) = \text{const.} \times x_1^d.
\]
\end{proof}
%%%%%%%%%%%%%%%%%%%%%%%%%%%%%%%%%%%%%%%%%%%%%%%%%%%%%%%%%%%%%%
Now let us assume that $\mathrm{rank} \, A = \mathrm{rank} \, B = 2$. 
Then the equation of $S$ is reduced to the following form
\begin{align*}
 \ell_0(x_0, x_1, x_2) &= 0 \\
 \ell_1(x_0, x_1, x_2) &= m_1(y_0, y_1, y_2) \\
                    0 &= m_2(y_0, y_1, y_2)
\end{align*} 
where $\ell_i$ and $m_i$ are linear forms. 
Note that $\ell_0(x)$ and $m_2(y)$ are uniquely determined by $S$,  
up to scalr multiplication. 
%%%%%%%%%%%%%%%%%%%%%%%%%%%%%%%%%%%%%%%%%%%%%%%%%%%%%%%%%%%%%% Lemma
\begin{lem}
We assume that $d = \deg X \geq 2$ and $S \subset X$. 
\\
{\rm (1)} \ If $S$ is given by the above form, $\ell_0(x)=0$ is a $d$-flex 
tangent of $C_1$ (we regard $[x_0:x_1:x_2]$ as coordinates of $\PP^2(\CC)$). 
The same is true for $m_2(y)=0$ and $C_2$. 
\\
{\rm (2)} \ Conversely, a pair of $d$-flex tangents of $C_1$ and $C_2$ 
gives $d$ planes of rank $2$ in $X$. 
\\
{\rm (3)} \ There exist exactly $d \nu_1 \nu_2$ planes of rank $2$ in $X$, 
where $\nu_i$ is the number of $d$-flex points of $C_i$.  
\end{lem}
%%%%%%%%%%%%%%%%%%%%%%%%%%%%%%%%%%%%%%%%%%%%%%%%%%%%%%%%%%%%%%% proof
\begin{proof}
(1) \ For simplicity, we take projectve coodinates as  
\begin{align*}
 \ell_0(x_0, x_1, x_2) = x_2, \quad
 \ell_1(x_0, x_1, x_2) = x_1, \quad m_1(y_0, y_1, y_2)=y_1, \quad
 m_2(y_0, y_1, y_2) = y_2
\end{align*}
(this is possible by the action of $\GL_3(\CC) \times \GL_3(\CC)$, since 
$\mathrm{rank} \, A = \mathrm{rank} \, B = 2$). Namely, we assume that $S$ 
is given by
\[
 x_2 = 0, \quad x_1 - y_1 = 0, \quad y_2 = 0.  
\]
By the assumption $S \subset X$, we see that 
$F_1(x_0, x_1, 0) - F_2(y_0, y_1, 0)$ is divided 
by $x_1 - y_1$. By lemma \ref{Lemma-div}, we have 
\[
 F_1(x_0, x_1,0) = cx_1^d, \quad F_2(y_0, y_1, 0) = cy_1^d
\] 
where $c$ is a constant. This implies that $x_2 = 0$ is a $d$-flex tangent 
of $C_1$ and $y_2 = 0$ is a $d$-flex tangent of $C_2$. 
\\
(2) \ Let $\ell_0(x) = 0$ (resp. $m_2(y) = 0$) be a $d$-flex tangent of  
$C_1$ (resp. $C_2$). 
There exist homogeneous polynomials $G_1,\, G_2$ of degree $d-1$ and linear forms 
$\ell_1,\, m_1$ such that
\[
 F_1(x) = \ell_0(x) G_1(x) + \ell_1(x)^d, \quad
F_2(y) = m_2(y) G_2(y) + m_1(y)^d.
\]
Since $F_1$ is irreducible, $\ell_0$ is not a scalar multiple of $\ell_1$. 
The same is true for $m_1$ and $m_2$. We see that a system of linear equations
\begin{align*} %\label{reduced}
 \ell_0(x_0, x_1, x_2) &= 0 \\
 \ell_1(x_0, x_1, x_2) &= \zeta m_1(y_0, y_1, y_2) \\
                    0 &= m_2(y_0, y_1, y_2)
\end{align*} 
gives a plane in $X$ for any  $\zeta \in \mu_d$. 
\\
(3) \ By the above construction, we see that the number of planes of 
rank $2$ in $X$ is $d \nu_1 \nu_2$.
\end{proof}
%%%%%%%%%%%%%%%%%%%%%%%%%%%%%%%%%%%%%%%%%%%%%%%%%%%%%%%%%%%%%%%%%%%%%%%%%%%%%%%%%
Proposition \ref{main-th1} follows from Lemma 2.1, 2.2 and  2.4. 
In the last of this section, we consider the case of $d=3$.
It is well known that a smooth cubic curve has nine inflection points.
Hence a smooth cubic 4-fold $X$ of this type contains $3 \times 9 \times 9 = 243$ 
planes of rank 2. In the case of $F_1 \sim F_2$, we have extra $|\Aut(F_1)|$ planes 
of rank 3. To see automorphisms of ternary cubic forms, the Hesse normal form
\begin{align} \label{Hesse}
 F(x_0, x_1, x_2) = x_0^3 + x_1^3 + x_2^3 - 3 \lambda x_0 x_1 x_2  
\end{align}
is usefull (see, e.g. \cite{BM17}). It defines a smooth cubic curve 
if $\lambda \notin \mu_3$, and we have the j-invariant 
\[
 j = 1728 \frac{\lambda^3 (\lambda^3 + 8)^3}{64(\lambda^3 - 1)^3} 
= 1728 \frac{g_2^3}{g_2^3 - 27 g_3^2}.
\]
In general, the group 
$\Aut(F) / \mu_3 \subset \mathrm{PGL}_3(\CC)$ is of order 18 generated by 
permutations of coordinates and an automorphism
\begin{align*}
 (x_0, x_1, x_2) \mapsto (x_0, \omega x_1, \omega^2 x_2) 
\end{align*}
where $\omega = e^{2 \pi i/3}$ (If we take a flex point as the orgin $O$, 
this is the group generated by the inversion and translations by 3-tosions). 
As is well known, we have extra automorphisms for $j = 0, 1728$. 
For example, if $\lambda = 0$ (and hence $j = 0$), we have an automorphism
\begin{align*}
 (x_0, x_1, x_2) \mapsto (x_0, x_1, \omega x_2) 
\end{align*}
of order 3. As an example for $j = 1728$, if $\lambda = 1 + \sqrt{3}$, we have 
an automorphism 
\begin{align*}
  M = -\frac{1}{\sqrt{3}} \begin{bmatrix} 1 & 1 & 1 \\ 
1 & \omega & \omega^2 \\ 1 & \omega^2 & \omega \end{bmatrix} \in \GL_3(\CC)
\end{align*}
of order 4. 
%%%%%%%%%%%%%%%%%%%%%%%%%%%%%%%%%%%%%%%%%%%%%%%%%%%%%%%%%%%%%%%%%%%%%%%%
\begin{rem}
For the Hesse normal form, we have the following flex points and tangent lines.   
%%%%%%%%%%%%%%%%%%%%%%%%%%%%%%%%%%%%%%%%%%%%%%%%%%%%%%%%%%%%%%%%%%%%%%%%%%%%%%%%%%
\begin{align*}
\begin{array}{c|c|| c|c} 
\text{$[0 : -1 : 1]$} & \lambda x_0 + x_1 + x_2 = 0 & 
\text{$[1: 0 : -\w^2]$} & x_0 + \lambda w^2 x_1 + \w x_2 = 0\\
\text{$[0 : -\w : 1]$} & \w \lambda x_0 + \w^2 x_1 + x_2 = 0 &
\text{$[-1 : 1 : 0]$} & x_0 + x_1 + \lambda x_2 = 0 \\
\text{$[0 : -\w^2: 1]$} & \w^2 \lambda x_0 + \w x_1 + x_2 = 0 &
\text{$[-\w : 1 : 0]$} & \w^2 x_0 + x_1 + \lambda \w x_2 = 0 \\
\text{$[1 : 0 : -1]$} & x_0 + \lambda x_1 + x_2 = 0 &
\text{$[-\w^2 : 1: 0]$} & \w x_0 + x_1 + \lambda \w^2 x_2 = 0 \\
\text{$[1 : 0 : -\w]$} & x_0 + \lambda \w x_1 + \w^2 x_2 = 0 \\
\end{array}
\end{align*}
%%%%%%%%%%%%%%%%%%%%%%%%%%%%%%%%%%%%%%%%%%%%%%%%%%%%%%%%%%%%%%%%%%%%%%%%%%%%%%%%% 
\end{rem}
%%%%%%%%%%%%%%%%%%%%%%%%%%%%%%%%%%%%%%%%%%%%%%%%%%%%%%%%%%%%%%%%%%%%%%%%%%%%%%%%%
\section{transcendental cycles}
\subsection{transcendental lattice}
%%%%%%%%%%%%%%%%%%%%%%%%%%%%%%%%%%%%%%%%%%%%%%%%%%%%%%%%%%%%%%%%%%%%%%%%%%%%%%%%%
We compute intersection numbers of topological 4-cycles in $X$ coming from 
$C_1 \times C_2$. Let us consider a divisor  
\begin{align*}
 V : F_1(x) = F_2(y) = 0 
\end{align*}
of $X$. It is singular along $C_1 \cup C_2$, and the projection
\begin{align*}
 p : V - C_1 \cup C_2 \longrightarrow C_1 \times C_2, \qquad
[x:y] \mapsto [x_0:x_1:x_2] \times [y_0:y_1:y_2] 
\end{align*}
gives a structure of $\CC^{*}$-bundle.  
Let $\pi : \widetilde{X} \rightarrow X$
be the blow up of $X$ along $C_1 \cup C_2$ and $\widetilde{V}$ be 
the strict transform of $V$. Then we have a $\PP^1$-bundle
$p : \widetilde{V} \rightarrow C_1 \times C_1$. 
%%%%%%%%%%%%%%%%%%%%%%%%%%%%%%%%%%%%%%%%%%%%%%%%%%%%%%%%%%%%%
%%%%%%%%%%%%%%%%%%%%%%%%%%%%%%%%%%%%%%%%%%%%%%%%%%%%%%%%%%%%% picture
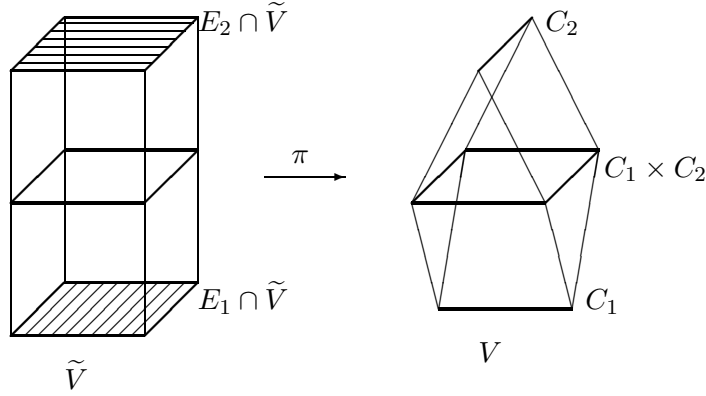
\begin{center}
\begin{picture}(270,135)(0,0)
%\put(0,0){\dashbox{5}(270,140)}
\put(20,-20){$\widetilde{V}$} \put(175,-10){$V$} 
%%%%
\thicklines
\multiput(0,0)(0,50){3}{\line(1,0){50}}
\multiput(20,20)(0,50){3}{\line(1,0){50}}
\multiput(0,0)(0,50){3}{\line(1,1){20}}
\multiput(50,0)(0,50){3}{\line(1,1){20}}
\thinlines
\multiput(0,0)(50,0){2}{\line(0,1){100}}
\multiput(20,20)(50,0){2}{\line(0,1){100}}
%%%%
\multiput(0,0)(5,0){10}{\line(1,1){20}}
\put(70,10){$E_1 \cap \widetilde{V}$}
\multiput(0,100)(3,3){7}{\line(1,0){50}}
\put(70,115){$E_2 \cap \widetilde{V}$}
%%%%
\put(105,65){$\pi$}
%\put(100,60){$\longrightarrow$}
\put(95,60){\vector(1,0){30}}
%%%%
\thicklines
\multiput(150,50)(20,20){2}{\line(1,0){50}}
\put(160,10){\line(1,0){50}} \put(175,100){\line(1,1){20}}
\multiput(150,50)(50,0){2}{\line(1,1){20}}
\thinlines
\multiput(150,50)(20,20){2}{\line(1,2){25}}
\multiput(200,50)(20,20){2}{\line(-1,2){25}}
\multiput(150,50)(50,0){2}{\line(1,-4){10}}
\multiput(160,10)(50,0){2}{\line(1,6){10}}
\put(215,10){$C_1$} \put(200,115){$C_2$}
\put(222,60){$C_1 \times C_2$} 
%%%%
\end{picture}
\end{center}
\vskip1cm
%%%%%%%%%%%%%%%%%%%%%%%%%%%%%%%%%%%%%%%%%%%%%%%%%%%%%%%%%%%%%%%%%%%%%%
Let $\iota : \widetilde{V} \rightarrow \widetilde{X}$ be the inclusion map. 
We have the following diagram
%%%%%%%%%%%%%%
\begin{align} \label{diagram}
\begin{CD}
\widetilde{V} @>\iota>> \ \widetilde{X} @>\pi>> X\\
@VpVV \\
C_1 \times C_2     
\end{CD}
\end{align}
%%%%%%%%%%%%%
and a morphism of integral Hodge structures
%%%%%%%%%%%%%%
\begin{align*} 
 (\pi \circ \iota)_* \circ p^* : 
H^2(C_1 \times C_2, \ZZ)(1) \longrightarrow H^4(X, \ZZ)
\end{align*}
%%%%%%%%%%%%%
where the pushforward map $(\pi \circ \iota)_*$ is defined via the Poincar\'e duality
(In the following, we often identify cohomology groups and homology groups). 
Our interset is the restriction on a K\"unneth component
\begin{align} \label{morphism}
\phi : 
  [H^1(C_1, \ZZ) \otimes H^1(C_2, \ZZ)](1) \longrightarrow H^4(X, \ZZ). 
\end{align}
%%%%%%%%%%%%%%%%%%%%%%%%%%%%%%%%%%%%%%%%%%%%%%%%%%%%%%%%%%%%%%%%%%%%%%%% Lemma 
\begin{lem}
Let $\gamma_i$ be a topological 1-cycle on $C_i$ and put
%%%%%%%%%%%%%%
\begin{align*}
  \Gamma = p^{-1}(\gamma_1 \times \gamma_2) 
\subset \widetilde{V} \subset \widetilde{X}. 
\end{align*} 
%%%%%%%%%%%%%%
For $[\Gamma] \in H^4(\widetilde{X}, \ZZ)$, 
we have $\pi^* \pi_* [\Gamma] = [\Gamma]$. 
\end{lem}
%%%%%%%%%%%%%%%%%%%%%%%%%%%%%%%%%%%%%%%%%%%%%%%%%%%%%%%%%%%%%%%%%%%%%% proof
\begin{proof}
It is enough to show in the case that $\gamma_1$ and $\gamma_2$ are smooth and 
isomorphic to a circle $S^1$ since general 1-cycles are homologous to $\ZZ$-linear 
combinations of them. Namely, we assume that $\Gamma$ is a $\PP^1(\CC)$-bundle 
over $S^1 \times S^1$.  
\\ \indent
Let $E_i = \pi^{-1}(C_i)$ be the exceptional divisor. First we show that 
$\Gamma$ and $E_i$ intersect transversally in $\widetilde{X}$. 
Since $C_1 \cap C_2 = \phi$ and the problem is local, we consider the 
blow up along only $C_1$. Then $\widetilde{X}$ is realized in 
$\PP^5(\CC) \times \PP^2(\CC)$ as 
\begin{align*}
F_1(x_0, x_1, x_2) = F_2(y_0, y_1, y_2), \qquad
 \mathrm{rank} \begin{bmatrix} y_0 & y_1 & y_2 \\ 
t_0 & t_1 & t_2 \end{bmatrix} = 1
\end{align*}
where $[t_0:t_1:t_2] \in \PP^2(\CC)$. In an affine open set given by 
$x_0 = 1$ and $t_0 = 1$ (that is $y_1 = t_1 y_0, \ y_2 = t_2 y_0$), we have
\begin{align*}
\widetilde{X} &: F_1(1, x_1, x_2) = y_0^d F_2(1, t_1, t_2)
\\
\widetilde{V} &: F_1(1, x_1, x_2) = F_2(1, t_1, t_2) = 0
\\
E_1 &: F_1(1, x_1, x_2) = y_0 = 0
\end{align*}
where $(x_1, x_2, y_0, t_1, t_2) \in \CC^5$.  The exceptional divisor $E_1$ is 
given by $y_0 = 0$ in $\widetilde{X}$, and $y_0$ is a (complex) local coordinate 
for the fiber direction of $\widetilde{V}$. We have same situations in other affine 
charts. From this, we see transversality of $E_1$ and $\Gamma$, that is, 
$T_xE_1 + T_x \Gamma = T_x \widetilde{X}$ for 
$x \in E_1 \cap \Gamma$. 
\\ \indent
Now we have $E_1 \cap \Gamma \cong \gamma_1 \times \gamma_2 = S^1 \times S^1$ 
(a smooth section of a fiber bundle $\Gamma \rightarrow S^1 \times S^1$), and 
\begin{align*}
 i_1^* [\Gamma] = [\Gamma \cap E_1] \in H_2(E_1, \ZZ) \cong H^4(E_1, \ZZ)
\end{align*}
where $i_1 : E_1 \rightarrow \widetilde{X}$ 
is the inclusion map. Note that the projection $\pi$ gives fibrations 
\begin{align*}
\begin{array}{cccl}
E_1 & \longrightarrow & C_1 & \ \ \PP^2(\CC) \text{-bundle} \\    
\cup & & \cup \\
(\Gamma \cap E_1) & \longrightarrow & \gamma_1 = S^1 & \ \ 
\gamma_2 \text{-bundle} \ (S^1 \text{-bundle}) 
\end{array} 
\end{align*}
and $\Gamma \cap E_1$ is the boudary of a solid torus in $E_1$. 
Therefore we have $i_1^*[\Gamma] = 0$, and the same is true for the 
pull-back to $H_2(E_2, \ZZ)$. 
This implies that $[\Gamma] = \pi^* \alpha$ for some $\alpha \in H^4(X, \ZZ)$ since 
there is an exact sequence
\begin{align*}
H^4(X, \ZZ) \longrightarrow H^4(\widetilde{X}, \ZZ) 
\longrightarrow H^4(E, \ZZ).
\end{align*}
Therefore we have $\pi^* \pi_* [\Gamma] = \pi^* \pi_* \pi^* \alpha = \pi^* \alpha 
= [\Gamma]$. 
\end{proof}
%%%%%%%%%%%%%%%%%%%%%%%%%%%%%%%%%%%%%%%%%%%%%%%%%%%%%%%%%%%%%%%%%%%%%%%%%%%%%%%
%%%%%%%%%%%%%%%%%%%%%%%%%%%%%%%%%%%%%%%%%%%%%%%%%%%%%%%%%%%%%%%%%%% Lemma
\begin{lem} \label{trans}
{\rm (1)} \ The morphism $\phi$ in {\rm (\ref{morphism})} satisfies  
\begin{align*}
\phi(\alpha) \cdot \phi(\beta) = -(\deg X)(\alpha \cdot \beta). 
\end{align*}
In particular, the image $\rm{Im} \, \phi$ is a sublattice of rank $4g^2$ with 
the intersection form
\begin{align*}
\begin{bmatrix}0 & d \\ d & 0\end{bmatrix} \oplus \cdots \oplus
\begin{bmatrix}0 & d \\ d & 0\end{bmatrix} \ (2g^2 \ \text{\rm{times}})
\end{align*}
where $d = \deg X$ and $g = \frac{(d-1)(d-2)}{2}$ is the genus of $C_i$.
\\ 
{\rm (2)} \ For a plane $S \subset X$ of rank $2$, the class $[S] \in H^4(X, \ZZ)$ 
is orthogonal to $\rm{Im} \, \phi$. 
\end{lem}
%%%%%%%%%%%%%%%%%%%%%%%%%%%%%%%%%%%%%%%%%%%%%%%%%%%%%%%%%%%%%%%%%%%%%%%%%%%%%%%
\begin{proof}
(1) \ For $\alpha, \ \beta \in H^1(C_1, \ZZ) \otimes H^1(C_2, \ZZ)$, 
we denote $p^{*} \alpha$ and $p^* \beta$ by $A$ and $B$. 
By the projection formula and the previous Lemma, we have
\begin{align*}
\phi(\alpha) \cdot \phi(\beta) &= \pi_*(\iota_*A) \cdot \pi_*(\iota_*B) \\ 
&= \pi_*((\pi^* \pi_* \iota_*A) \cdot \iota_*B)
= \pi_*(\iota_*A \cdot \iota_*B).
\end{align*}
We can compute this excess intersection by
\begin{align*}
 (\iota_*A \cdot \iota_*B)_{\widetilde{X}} 
= (A \cdot B \cdot \iota^*[\widetilde{V}])_{\widetilde{V}} 
= (p^{*} \alpha \cdot p^{*} \beta \cdot \iota^*[\widetilde{V}])_{\widetilde{V}} 
\end{align*}
where $(- \cdot -)_M$ represents the intersection product in $M$, and
$[\widetilde{V}] \in H^2(\widetilde{X}, \ZZ)$ is the class of 
$\widetilde{V}$. Let $F \cong \PP^1(\CC)$ be a fiber of $\widetilde{V}$. 
Then we have
\begin{align*}
 p^{*} \alpha \cdot p^{*} \beta = (\alpha \cdot \beta)_{C_1 \times C_2} \cdot [F]
\end{align*}
and hence
\begin{align*}
 \phi(\alpha) \cdot \phi(\beta) 
= (\alpha \cdot \beta) \cdot ([F] \cdot \iota^*[\widetilde{V}]).  
\end{align*}
Lastly, let $[H] \in H^2(X, \ZZ)$ be the hyperplane class and $E_i$ 
be exceptinal divisors $\pi^{-1}(C_i)$. We have 
%%%%%%%%%%%%%%%%%%%
\begin{align*}
[\widetilde{V}] = \pi^*[V] - d[E_1] - d[E_2]
= d(\pi^*[H] - [E_1] - [E_2]). 
\end{align*}
By calculation with local coordinates in the previous Lemma, we see that
\begin{align*}
[F] \cdot \pi^*[H] = 1, \qquad [F] \cdot \pi^*[E_i] = 1. 
\end{align*}
Therefore we have $[F] \cdot \iota^*[\widetilde{V}] = -d$, and we conclude the formula. 
Since the intersection form on $H^1(C_1, \ZZ) \otimes H^1(C_2, \ZZ) \cong \ZZ^{4g^2}$ is
\begin{align*}
\begin{bmatrix}0 & 1 \\ 1 & 0\end{bmatrix} \oplus \cdots \oplus
\begin{bmatrix}0 & 1 \\ 1 & 0\end{bmatrix} \ (2g^2 \ \text{\rm{times}}),
\end{align*}
we complete the proof of (1).
\\
(2) \ We may assume that a topological 1-cycle $\gamma_i$ on $C_i$ does not pass through 
any $d$-flex points on $C_i$, and that a plane $S \subset X$ of rank 2 is disjoint from 
$\mathrm{Im} \, \phi$. 
\end{proof}
%%%%%%%%%%%%%%%%%%%%%%%%%%%%%%%%%%%%%%%%%%%%%%%%%%%%%%%%%%%%%%%%%%%%%%%%%%%%%%%
%%%%%%%%%%%%%%%%%%%%%%%%%%%%%%%%%%%%%%%%%%%%%%%%%%%%%%%%%%%%%%%%%%%%%%%%%%%%%%%%%
\section{cubic 4-folds}
%%%%%%%%%%%%%%%%%%%%%%%%%%%%%%%%%%%%%%%%%%%%%%%%%%%%%%%%%%%%%%%%%%%%%%%%%%%%%%%%%
\subsection{cubic 4-folds}
Let $X \subset \PP^5(\CC)$ be a smooth cubic hypersurface. We have
\begin{align*}
  h^{4,0}(X) = 0, \quad h^{3,1}(X) = 1, \quad h^{2,2}(X) = 21
\end{align*}
and the integral middle cohomology $H^4(X, \ZZ)$ is the unimodular lattice 
$(+1)^{\oplus 21} \oplus (-1)^{\oplus 2}$.  
The integral Hodge conjecture is known to be valid, that is, 
$H^{2,2}(X) \cap H^4(X, \ZZ)$ is spanned by algebraic cycles (\cite{V07}). 
Let $T_X \subset H^4(X, \ZZ)$ be the orthogonal complement of algebraic cycles. 
We call it the transcendental lattice of $X$. In the following, we show that 
$T_X$ for a general cubic fourfold $X$ given by $F(X) = F(y)$ is 
\begin{align*}
\mathrm{Im} \, \phi = \begin{bmatrix}0 & 3 \\ 3 & 0\end{bmatrix} \oplus 
\begin{bmatrix}0 & 3 \\ 3 & 0\end{bmatrix}. 
\end{align*}
%%%%%%%%%%%%%%%%%%%%%%%%%%%%%%%%%%%%%%%%%%%%%%%%%%%%%%%%%%%%%%%%%%%%%%%%%%%%%%%%%
%%%%%%%%%%%%%%%%%%%%%%%%%%%%%%%%%%%%%%%%%%%%%%%%%%%%%%%%%%%%%%%%%%%%%%%%%%%%%%%%%
\subsection{Fermat cubic 4-fold}
In \cite{DS16}, Degtyarev and Shimada studied the sublattice of the middle homology 
group of Fermat varieties generated by the classes of linear subspaces. 
They gave an algebraic (or rather combinatorial) criterion 
for the primitivity of this submodule. To apply this result, we take
%%%%%%%%%%%%%%%
\begin{align*}
F_1 = z_0^3 + z_1^3 + z_2^3, \quad 
F_2 = -(z_3^3 + z_4^3 + z_5^3).
\end{align*}
%%%%%%%%%%%%%%%
Namely, we consider the Fermat cubic 4-fold
\begin{align*}
 X_0 : F_1 - F_2 = (z_0^3 + z_1^3 + z_2^3) + (z_3^3 + z_4^3 + z_5^3) = 0.  
\end{align*}
According to Degtyarev and Shimada, let $\mathcal{K}$ be the set of indices
%%%%%%%%%%%%
\begin{align*}
J_1 = [01|23|45], \quad J_2 = [01|24|35], \quad J_3 = [02|13|45], 
\quad J_4 = [02|14|35]. 
\end{align*}
%%%%%%%%%%%%
For $J = [j_0,k_0 | j_1,k_1 | j_2,k_2] \in \mathcal{K}$ and 
$\beta = (\beta_0, \beta_1, \beta_2) \in \mathcal{B} = (\mu_3)^3$, we define a
plane
\begin{align*}
 L_{J, \beta} \ : \ z_{k_0} = -\beta_0 z_{j_0}, \quad z_{k_1} = -\beta_1 z_{j_1}, \quad 
z_{k_2} = -\beta_0 z_{j_2}  
\end{align*}
in $X$. Note that these planes are of rank $2$, that is, induced by flex tangents of $C_i$. 
Let $\mathcal{L}_{\mathcal{K}}$ be the submodule of $H^4(X_0, \ZZ)$ generated by 
$[L_{J, \beta}]$. By Theorem 1.1 in \cite{DS16}, the torsions of the 
quotient module 
$H^4(X_0, \ZZ)/\mathcal{L}_{\mathcal{K}}$ is isomorphic to the torsions of
%%%%%%%%%%%%%%%%%%%%%%%%%%%%%%%%%%%%%%%%%% 
\begin{align} \label{DS-module}
 \ZZ[t_1, t_2, t_3, t_4, t_5]/
(t_i^2 + t_i + 1,\ \rho_J \ |\  i=1,\dots,5, \ J \in \mathcal{K}),
\end{align}
%%%%%%%%%%%%%%%%%%%%%%%%%%%%%%%%%%%%%%%%%
where
\begin{align*}
\rho_J = (1 + t_{j_1} + t_{j_1} t_{k_1}) (1 + t_{j_2} + t_{j_2} t_{k_2}). 
\end{align*}
%%%%%%%%%%%%%%%%%%%%%%%%%%%%%%%%%%%%%%%%%%%%%%%%%%%%%%%%%%%%%%%%%%%%%%%%%%%%%
%%%%%%%%%%%%%%%%%%%%%%%%%%%%%%%%%%%%%%%%%%%%%%%%%%%%%%%%%%%%%%%%%%% Lemma
\begin{lem}
The module {\rm (\ref{DS-module})} is torsion-free. Therefore 
the sublattice $\mathcal{L}_{\mathcal{K}}$ is primitive in $H^4(X_0, \ZZ)$.  
\end{lem}
%%%%%%%%%%%%%%%%%%%%%%%%%%%%%%%%%%%%%%%%%%%%%%%%%%%%%%%%%%%%%%%%%%%%%%%%%%%%%%
\begin{proof}
Note that the quotient
\[
 \ZZ[t_1, t_2, t_3, t_4, t_5]/(t_i^2 + t_i + 1 \ |\  i=1,\dots,5)
\]
is a free module of rank $2^5$, generated by monomials
$t^q = t_1^{q_1} t_2^{q_2} t_3^{q_3} t_4^{q_4} t_5^{q}$ with $q_i \in \{0,1\}$. 
On the other hand, polynomials $\rho_J$ are sums of these monomials, 
and coefficients of monomials
\[
1, \quad t_3, \quad t_1 t_3, \quad t_1 t_3 t_5
\]
in $\rho_J$ are given in the following table.
\begin{align*}
\begin{array}{c|cccc}
 & 1 & t_3 & t_1 t_3 & t_1 t_3 t_5 \\
\hline
\rho_{01|23|45} = (1 + t_2 + t_2 t_3) (1 + t_4 + t_4 t_5) & 1 & 0 & 0 & 0 \\
\rho_{01|24|35} = (1 + t_2 + t_2 t_4) (1 + t_3 + t_3 t_5) & 1 & 1 & 0 & 0 \\
\rho_{02|13|45} = (1 + t_1 + t_1 t_3) (1 + t_4 + t_4 t_5) & 1 & 0 & 1 & 0 \\
\rho_{02|14|35} = (1 + t_1 + t_1 t_4) (1 + t_3 + t_3 t_5) & 1 & 1 & 1 & 1 \\
\end{array}
\end{align*}
This $4 \times 4$ matrix is unimodular. Hence $\oplus \ZZ \rho_J$ is 
primitive in $\oplus \ZZ t^q$ and (\ref{DS-module}) is free.  
\end{proof}
%%%%%%%%%%%%%%%%%%%%%%%%%%%%%%%%%%%%%%%%%%%%%%%%%%%%%%%%%%%%%%%%%%%%%%%%%%%%%
To compute the primitive sublattice $\mathcal{L}_{\mathcal{K}}$ explicitly, 
we note that
%%%%%%%%%%%%%%%%%%%%%%%%%%%%%%%%%%%%%%%%%%%%%%%%%%%%%%%%%%%%%%%%%%% Lemma
\begin{lem}
Let $X$ be a smooth cubic fourfold, $S$ and $T$ be planes in $X$.  
Then we have 
\begin{align*}
 [S] \cdot [T] 
= \begin{cases} 0 \quad (S \cap T = \phi) \\ 1 \quad (S \cap T = \text{a point}) \\ 
-1 \quad (S \cap T = \PP^1) \\ 3 \quad (S = T) \end{cases}
\end{align*}
\end{lem}
%%%%%%%%%%%%%%%%%%%%%%%%%%%%%%%%%%%%%%%%%%%%%%%%%%%%%%%%%%%%%%%%%%%%%% proof
\begin{proof}
The first and the second case are obvious. In the case of $S=T$, the self-intersection 
number $[S] \cdot [S]$ is given by the 2nd Chern class $c_2(N_{S/X})$ of the normal bundle 
$N_{S/X}$.  
If $S \cap T = \PP^1$, the intersection number is given by the coefficient of $t$ 
in the power series
\begin{align*}
 \frac{c_t(N_{S/X}) c_t(N_{T/X})}{c_t(N_{\PP^1 / X})}. 
\end{align*} 
From $c_t(N_{\PP^2 /X}) = 1+ 3t^2$ and $c_t(N_{\PP^1 /X}) = 1+ t$, we obtain intersection 
numbers.
\end{proof}
%%%%%%%%%%%%%%%%%%%%%%%%%%%%%%%%%%%%%%%%%%%%%%%%%%%%%%%%%%%%%%%%%%%%%%%%%
Therefore compuations of intersection numbers $[L_{I, \alpha}] \cdot [L_{J, \beta}]$ 
are reduced to that of ranks of linear equations for $L_{I, \alpha} \cap L_{J, \beta}$. 
Using a computer, we can show that  
%%%%%%%%%%%%%%%%%%%%%%%%%%%%%%%%%%%%%%%%%%%%%%%%%%%%%%%%%%%%%%%%%%%%%%%%%
%%%%%%%%%%%%%%%%%%%%%%%%%%%%%%%%%%%%%%%%%%%%%%%%%%%%%%%%%%%%% Proposition
\begin{lem} \label{S19}
The sublattice $\mathcal{L}_{\mathcal{K}} \subset H^4(X_0,\ZZ)$ is of rank $19$, 
and its discriminant is $81$. Moreover we have 
$(\mathcal{L}_{\mathcal{K}})^{\perp} = \mathrm{Im} \, \phi$, and hence 
$\mathrm{Im} \, \phi$ is a primitive sublattice of $H^4(X_0,\ZZ)$.  
\end{lem}
%%%%%%%%%%%%%%%%%%%%%%%%%%%%%%%%%%%%%%%%%%%%%%%%%%%%%%%%%%%%%%%%%%%%%%%%
\begin{proof}
Let us consider 19 planes $S_1, \dots, S_{19}$ in the following table. 
%%%
\begin{align*}
\begin{array}{|c|c|c||c|c|c||c|c|c||c|c|c|}
\hline
L_{J, \beta} & J & \beta & L_{J, \beta} & J & \beta & L_{J, \beta} & J & \beta &
L_{J, \beta} & J & \beta \\
\hline
%%%%%%%
S_1 & J_1 & (1, 1, 1) & S_6 & J_1 & (\w, 1, 1) &
S_{11} & J_2 & (1, 1, \w) & S_{16} & J_3 & (1, \w, 1) \\
%%%%%%%
S_2 & J_1 & (1, 1, \w) & S_7 & J_1 & (\w, 1, \w) &
S_{12} & J_2 & (\w, 1, 1) & S_{17} & J_3 & (1, \w, \w) \\
%%%%%%% 
S_3 & J_1 & (1, 1, \w^2) & S_8 & J_1 & (\w, \w, 1) &
S_{13} & J_2 & (\w, 1, \w) & S_{18} & J_4 & (1, 1, 1) \\
%%%%%%%
S_4 & J_1 & (1, \w, 1) & S_9 & J_1 & (\w, \w, \w) &
S_{14} & J_3 & (1, 1, 1) & S_{19} & J_4 & (1, 1, \w) \\
%%%%%%%
S_5 & J_1 & (1, \w, \w) & S_{10} & J_2 & (1, 1, 1) &
S_{15} & J_3 & (1, 1, \w) & & & \\
%%%%%%%
\hline
\end{array}
\end{align*}
%%%%%%%%%%%%%%%
Using the previous Lemma, we can compute the intersection matrix $M$ of 
these 19 planes (see Appendix) and we have $\det M = 81$. 
Since $\mathcal{L}_{\mathcal{K}}$ is orthogonal to $\rm{Im} \, \phi$ 
(which is of rank $4$) in Lemma \ref{trans}, we have 
\begin{align*}
  (\left<[S_1], \dots, [S_{19}] \right>_{\ZZ} \oplus \mathrm{Im} \, \phi) 
\otimes \QQ = H^4(X_0, \QQ).
\end{align*}
Therefore these 19 planes form  a basis of $\mathcal{L}_{\mathcal{K}} \otimes \QQ$. 
To complete the proof, we need to show that 108 planes $L_{J, \beta}$ are represented by 
$\ZZ$-linear combinations of $S_1, \dots, S_{19}$ in $H^4(X_0, \ZZ)$. 
Using a computer, we can check this by the following way. For $J \in \mathcal{K}$ 
and $\beta \in \mathcal{B}$, let $M_{J, \beta}$ be the intersection matrix of 
20 planes $S_1, \dots, S_{19}, L_{J, \beta}$. We can check that $\det M_{J, \beta} = 0$ and 
the eigenspace of $0$ is 1-dimensional. Moreover we have an integral eigenvector of 
the form
\begin{align*}
  (m_1, m_2, \dots, m_{19},1) \in \ZZ^{20},
\end{align*}
and then
\begin{align*}
  m_1[S_1] + m_2[S_2] + \dots + m_{19} [S_{19}] + [L_{J, \beta}] = 0
\end{align*}
in $H^4(X_0,\ZZ)$, since it is orthogonal to $[S_i]$ and $\mathrm{Im} \, \phi$. 
\end{proof}
%%%%%%%%%%%%%%%%%%%%%%%%%%%%%%%%%%%%%%%%%%%%%%%%%%%%%%%%%%%%%%%%%%%%%%%%%%%%
Note that our cubic fourfolds $X$ are obtained as deformations of 
the Fermat cubic $X_0$, with diagram \ref{diagram}. Therefore 
$\mathrm{Im} \, \phi$ is primitive in $H^4(X, \ZZ)$ for other $X$. 
Now Theorem \ref{main-th2} follows from facts that $\phi$ is a morphism 
of Hodge structures and $\phi_{\CC}$ maps
\begin{align*}
[H^{1,0}(C_1) \otimes H^{1,0}(C_2)] \oplus [H^{0,1}(C_1) \otimes H^{0,1}(C_2)]
\end{align*}
to $H^{3,1}(X) \oplus H^{1,3}(X)$, and $H^1(C_1, \QQ) \otimes H^1(C_2, \QQ)$ is 
indecomposable as a Hodge structure for general $C_1$ and $C_2$. 
%%%%%%%%%%%%%%%%%%%%%%%%%%%%%%%%%%%%%%%%%%%%%%%%%%%%%%%%%%%%%%%%%%%%%%%%%%%%%%% 
\appendix
\section{Intersection matrix}
%%%%%%%%%%%%%%%%%%%%%%%%
Let $\mathbb{I}$ be a $19 \times 19$ matrix all of whose entries are $1$. 
Let $M$ be the intersection matrix of $S_1, \dots, S_{19}$ in the proof of 
Lemma \ref{S19}. Then the matrix $M + \mathbb{I}$
(to avoid using the minus sign, we added $\mathbb{I}$) is given by 
\small
\[
M + \mathbb{I} =
\left(
\begin{array}{ccccc|ccccc|ccccc|cccc}
 4 & 0 & 0 & 0 & 2 & 0 & 2 & 2 & 1 & 0 & 2 & 2 & 1 & 0 & 2 & 2 & 1 & 2 & 1 \\
 0 & 4 & 0 & 2 & 0 & 2 & 0 & 1 & 2 & 2 & 0 & 1 & 2 & 2 & 0 & 1 & 2 & 1 & 2 \\
 0 & 0 & 4 & 2 & 2 & 2 & 2 & 1 & 1 & 2 & 2 & 1 & 1 & 2 & 2 & 1 & 1 & 1 & 1 \\
 0 & 2 & 2 & 4 & 0 & 2 & 1 & 0 & 2 & 2 & 2 & 1 & 1 & 2 & 1 & 0 & 2 & 1 & 1 \\
 2 & 0 & 2 & 0 & 4 & 1 & 2 & 2 & 0 & 0 & 2 & 2 & 1 & 1 & 2 & 2 & 0 & 2 & 1 \\
\hline
 0 & 2 & 2 & 2 & 1 & 4 & 0 & 0 & 2 & 2 & 1 & 0 & 2 & 2 & 1 & 2 & 1 & 1 & 2 \\
 2 & 0 & 2 & 1 & 2 & 0 & 4 & 2 & 0 & 1 & 2 & 2 & 0 & 1 & 2 & 1 & 2 & 1 & 1 \\
 2 & 1 & 1 & 0 & 2 & 0 & 2 & 4 & 0 & 1 & 1 & 2 & 2 & 0 & 2 & 2 & 1 & 2 & 1 \\
 1 & 2 & 1 & 2 & 0 & 2 & 0 & 0 & 4 & 2 & 1 & 0 & 2 & 2 & 0 & 1 & 2 & 1 & 2 \\
 0 & 2 & 2 & 2 & 0 & 2 & 1 & 1 & 2 & 4 & 0 & 0 & 2 & 2 & 1 & 1 & 2 & 0 & 2 \\
\hline
 2 & 0 & 2 & 2 & 2 & 1 & 2 & 1 & 1 & 0 & 4 & 2 & 0 & 1 & 2 & 1 & 1 & 2 & 0 \\
 2 & 1 & 1 & 1 & 2 & 0 & 2 & 2 & 0 & 0 & 2 & 4 & 0 & 1 & 2 & 1 & 1 & 2 & 1 \\
 1 & 2 & 1 & 1 & 1 & 2 & 0 & 2 & 2 & 2 & 0 & 0 & 4 & 1 & 1 & 2 & 1 & 1 & 2 \\
 0 & 2 & 2 & 2 & 1 & 2 & 1 & 0 & 2 & 2 & 1 & 1 & 1 & 4 & 0 & 0 & 2 & 0 & 2 \\
 2 & 0 & 2 & 1 & 2 & 1 & 2 & 2 & 0 & 1 & 2 & 2 & 1 & 0 & 4 & 2 & 0 & 2 & 0 \\
\hline
 2 & 1 & 1 & 0 & 2 & 2 & 1 & 2 & 1 & 1 & 1 & 1 & 2 & 0 & 2 & 4 & 0 & 2 & 2 \\
 1 & 2 & 1 & 2 & 0 & 1 & 2 & 1 & 2 & 2 & 1 & 1 & 1 & 2 & 0 & 0 & 4 & 0 & 2 \\
 2 & 1 & 1 & 1 & 2 & 1 & 1 & 2 & 1 & 0 & 2 & 2 & 1 & 0 & 2 & 2 & 0 & 4 & 0 \\
 1 & 2 & 1 & 1 & 1 & 2 & 1 & 1 & 2 & 2 & 0 & 1 & 2 & 2 & 0 & 2 & 2 & 0 & 4 \\
\end{array}
\right).
\]
%%%%%%%%%%%%%%%%%%%%%%%%%%%%%%%%%%%%%%%%%%%%%%%%%%%%%%%%%%%%%%%%%%%%%%%%%%%%%%%%
\normalsize
%%%%%%%%%%%%%%%%%%%%%%%%%%%%%%%%%%%%%%%%%%%%%%%%%%%%%%%%%%%%%%%%%%%%%%%%%%%%%%%%%%

\end{document}